\documentclass{svjour3-modified}  

\usepackage{amsmath,amssymb,xcolor}
\usepackage{enumitem}
\usepackage[margin=3.5cm]{geometry}
\usepackage[utf8]{inputenc}

\newcommand{\conv}{\operatorname{conv}}
\newcommand{\floor}[1]{\lfloor#1\rfloor}
\newcommand{\ceil}[1]{\lceil#1\rceil}

\newcommand{\rank}{\operatorname{rank}}

\newcommand{\R}{\mathbb{R}}
\newcommand{\Z}{\mathbb{Z}}

\title{Distances of optimal solutions of mixed-integer programs}
\institute{Joseph Paat, Robert Weismantel, Stefan Weltge \at Institute for Operations Research, ETH Zurich, Switzerland}
\author{Joseph Paat \and Robert Weismantel \and Stefan Weltge}

\begin{document}

\maketitle

\begin{abstract}
A classic result of Cook et al. (1986) bounds the distances between optimal solutions of mixed-integer linear programs and optimal solutions of the corresponding linear relaxations.
Their bound is given in terms of the number of variables and a parameter $ \Delta $, which quantifies sub-determinants of the underlying linear inequalities.
We show that this distance can be bounded in terms of $ \Delta $ and the number of integer variables rather than the total number of variables.
To this end, we make use of a result by Olson (1969) in additive combinatorics and demonstrate how it implies feasibility of certain mixed-integer linear programs.
We conjecture that our bound can be improved to a function that only depends on $ \Delta $, in general.
\end{abstract}


\section{Introduction}
In this paper, we consider the question of bounding distances between optimal solutions of mixed-integer linear programs that only differ in the sets of integer constraints.
Let $ A \in \Z^{m\times n} $, $ b \in \Z^m $, and $ c \in \R^n $.
For $ I \subseteq \{1, \dotsc, n\} =: [n] $, consider the mixed-integer linear program
\begin{equation}\label{eqIMIP}
    \max \, \{ c^\intercal x : Ax \le b, \, x_i \in \Z \text{ for all } i \in I \}.\tag{$I$-MIP}
\end{equation}
Notice that ($[n]$-MIP) describes a pure integer linear program and ($\emptyset$-MIP) describes its standard relaxation, which is a linear program.
Assuming that ($I$-MIP) has an optimal solution for every $ I \subseteq [n] $, we are interested in the following classic question.
Given $ I,J \subseteq [n] $ and an optimal solution for ($ I $-MIP), how close is the nearest optimal solution for ($ J $-MIP)?
We measure distance with respect to the maximum norm $ \| \cdot \|_\infty $ and focus on bounds that only depend on $ A $, $ I $, and $ J $.

One of the first explicit attempts to obtain such distance bounds can be found in the work of Blair and Jeroslow~\cite{BJ1977,BJ1979}, which was later improved by Cook~et~al.~\cite{CGST1986}.
To state their result, let $ \Delta = \Delta(A) $ denote the largest absolute value of any determinant of a square submatrix of $ A $.
\begin{theorem}[{Cook et al. (1986), see~\cite[Thm. 1 \& Rem. 1]{CGST1986}}] \label{thm0toN}
    Let $ I,J \subseteq [n] $ such that~($J$-MIP) has an optimal solution and either $ I = \emptyset $ or $ J = \emptyset $.
    For every optimal solution $ w $ of~($I$-MIP), there exists an optimal solution $z$ of~($J$-MIP) such that $\|w-z\|_{\infty} \le n\Delta$.
\end{theorem}
Observe that Theorem~\ref{thm0toN} only refers to situations in which one of the programs considered is the linear program.
For general $ I, J \subseteq [n] $, a bound of $ 2n\Delta $ is obtained using the triangle inequality.
However, for any choice of $I$ and $J$, their resulting bound depends on $ \Delta $ and the total number of variables $ n $.
The main purpose of this paper is to strengthen this dependence by showing that $ n $ can be replaced by the number of integer variables that appear in the two programs.
\begin{theorem} \label{thmMain1}
    Let $I,J\subseteq [n] $ with $ I \ne J $ such that~($J$-MIP) has an optimal solution.
    For every optimal solution $ w $ of~($I$-MIP), there exists an optimal solution $z$ of ($J$-MIP) such that $ \| w - z \|_{\infty} < |I \cup J| \Delta $.
\end{theorem}
To obtain our result, we make use of a result in additive combinatorics by Olson~\cite{Olson69} that determines the so-called Davenport constant of certain abelian groups.
We show how Olson's result implies that mixed-integer linear programs of a certain structure have non-zero solutions.
More precisely, we establish the following result, which may be of independent interest.
\begin{lemma} \label{lemIntFeasbility}
    Let $ d,k \in \Z_{\ge 1} $, $ u^1, \dotsc, u^k \in \Z^d $, and $ \alpha_1,\dotsc,\alpha_k \ge 0 $.
    If $ \sum_{i=1}^k \alpha_i \ge d $, then there exist $ \beta_i \in [0,\alpha_i] $ for $ i=1,\dotsc,k $ such that not all $\beta_1, \dots, \beta_k$ are zero and $ \sum_{i=1}^k \beta_i u^i \in \Z^d $.
\end{lemma}
While the bound in Theorem~\ref{thmMain1} depends on the number of integer variables, we are not aware of any pairs of MIPs for which distances between optimal solutions cannot be bounded just in terms of $ \Delta $.
For this reason, we state the following conjecture.
\begin{conjecture} \label{conjMain}
    There exists a function $ f : \Z_{\ge 1} \to \R $ such that the following holds.
    Let $I,J \subseteq [n] $ such that~($J$-MIP) has an optimal solution.
    For every optimal solution $ w $ of~($I$-MIP), there exists an optimal solution $z$ of ($J$-MIP) such that $ \| w - z \|_{\infty} \le f(\Delta) $.
\end{conjecture}
In fact, we believe that $ f $ can be chosen to be a linear function.
We conclude this paper by discussing this conjecture and providing some conditions under which it holds.

\subsubsection*{Related work}
Theorem~\ref{thm0toN} was extended to the case of separable quadratic objective functions in~\cite[Theorem 2]{GSK1988} and later to the more general case of separable convex objective functions in~\cite[Theorem 3.3]{HS1990} and~\cite[Theorem 1]{WM1991}.
In~\cite{BAL1995}, it was shown that a closer analysis of the parameter $\Delta$ can lead to strengthened results for certain choices of the matrix $A$.
The proofs of these generalized results are similar to the proof of Theorem~\ref{thm0toN}, albeit with additional analysis. 
The proof of Theorem~\ref{thmMain1} that is presented in this paper is also similar to the proof of Theorem~\ref{thm0toN}, and, consequently, the result can be generalized to the settings of~\cite{BAL1995},~\cite{GSK1988},~\cite{HS1990}, and~\cite{WM1991} using the techniques presented therein.
However, in order to highlight the importance of the ideas developed in this paper, we prove Theorem~\ref{thmMain1} for linear objective functions and omit the additional analysis required for these generalizations. 
We reemphasize that we study how the parameters $ \Delta $, $  I $, and $ J $ affect distance of mixed-integer programs in inequality form.
For recent developments on how other parameters affect the distance of integer linear programs in standard form, see, e.g.~\cite{EW2018}.

Interestingly, the Davenport constant was previously used in~\cite{AND2015} in the context of the dijoins and Woodall's conjecture.

\subsubsection*{Outline}
We start by reviewing parts of the proof of Cook~et~al.~\cite{CGST1986} in Section~\ref{secMain1} and show how Lemma~\ref{lemIntFeasbility} can be applied to obtain Theorem~\ref{thmMain1}.
In Section~\ref{secDavenport}, we discuss the Davenport constant and the mentioned result by Olson~\cite{Olson69}, which allows us to prove Lemma~\ref{lemIntFeasbility}.
Finally, Section~\ref{SecDelta} contains a discussion of Conjecture~\ref{conjMain}.


\section{The proof of Theorem~\ref{thmMain1}} \label{secMain1}
Our proof of Theorem~\ref{thmMain1} follows the strategy developed by Cook~et~al.~\cite{CGST1986}, but differs in some parts in order to (i) be able to compare solutions of ($I$-MIP) and ($J$-MIP) with $ I,J \ne \emptyset $ directly, and to (ii) improve the bound of Theorem~\ref{thm0toN}.
For instance, we bypass the use of strong linear programming duality in~\cite{CGST1986}, which restricted one of the sets $ I, J $ to be the empty set.
\begin{proof}[of Theorem~\ref{thmMain1}]
    Without loss of generality, we assume that $ I \cup J = [d] $, where $ d \in [n] $.
    Let $ w \in \R^n $ be an optimal solution of ($I$-MIP).
    Choose any $ \tilde{z} \in \R^n $ that is an optimal solution of ($J$-MIP) and define $ y := \tilde{z} - w $.
    Partitioning the rows of $ A $ into submatrices $ A_1, A_2 $ such that $ A_1 y < 0 $ and $ A_2 y \ge 0 $, we define the cone
    \[
        C := \{ x \in \R^n : A_1 x \le 0, \, A_2 x \ge 0 \}.
    \]
    Note that $ C $ is defined by an integral matrix arising from $ A $ by multiplying some of its rows by $ -1 $.
    By standard arguments involving Cramer's rule, there exist integer vectors $ v^1,\dotsc,v^k \in \Z^n $ such that
    \[
        C = \{ \lambda_1 v^1 + \dotsb + \lambda_k v^k : \lambda_1,\dotsc,\lambda_k \ge 0 \},
    \]
    and $ \|v^i\|_\infty \le \Delta $ for $ i = 1,\dotsc,k $.
    Observe that $ y \in C $, and hence, there exist $ \lambda_1,\dotsc,\lambda_k \ge 0 $ such that
    \[
        y = \lambda_1 v^1 + \dotsb + \lambda_k v^k.
    \]
    Consider the set
    \[
        G := \{ (\bar\gamma_1,\dotsc,\bar\gamma_k) : \bar\gamma_i \in [0,\lambda_i] \text{ for all } i \in [k], \, \bar\gamma_1 v^1 + \dotsb + \bar\gamma_k v^k \in \Z^d \times \R^{n - d} \},
    \]
    which is non-empty (it contains the all-zero vector) and compact.
    Thus, there exists some $ (\gamma_1,\dotsc,\gamma_k) \in G $ maximizing $ \bar\gamma_1 + \dotsb + \bar\gamma_k $ over $ G $.
    Recalling that $ y = \tilde{z} - w = \sum_{i = 1}^k \lambda_i v^i $, we define the vectors 
    \[
        z := \tilde{z} - \sum_{i=1}^k \gamma_i v^i = w + \sum_{i=1}^k (\lambda_i - \gamma_i) v^i 
    \]
    and
    \[
        \tilde{w} := w + \sum_{i=1}^k \gamma_i v^i = \tilde{z} - \sum_{i=1}^k (\lambda_i - \gamma_i) v^i.
    \]

    First, we claim that $ z $ is feasible for ($ J $-MIP) and $ \tilde{w} $ is feasible for ($ I $-MIP).
    To see this, observe that the coordinates of $ z $ indexed by $ J $ are integer since this is the case for $ \tilde{z} $, $(\gamma_1, \dotsc, \gamma_k)\in G$, and $ J \subseteq [d] $.
    Similarly, the coordinates of $ \tilde{w} $ indexed by $ I $ are integer.
    Furthermore, by the definition of $ C $, we see that
    \[
        \begin{array}{rclclcl}
            A_1z&=&A_1w+\sum_{i=1}^k(\lambda_i-\gamma_i)A_1v^i &\le& A_1w&\le&b_1\\
            A_2z&=&  A_2\tilde{z}-\sum_{i=1}^k\gamma_i A_2v^i &\le&  A_2\tilde{z}&\le&b_2\\
            A_1\tilde{w}&=&A_1w+\sum_{i=1}^k\gamma_iA_1v^i&\le& A_1w&\le&b_1\\
            A_2\tilde{w}&=&A_2\tilde{z}-\sum_{i=1}^k(\lambda_i-\gamma_i)A_2v^i &\le& A_2\tilde{z} &\le&b_2,
        \end{array}
    \]
    which shows that $ Az \le b $ and $ A\tilde{w} \le b $.

    Second, we claim that $ z $ is optimal for ($ J $-MIP).
    Indeed, since $ w $ is optimal for ($ I $-MIP), we must have
    \[
        c^\intercal w \ge c^\intercal \tilde{w} = c^\intercal w + c^\intercal \left(\sum_{i=1}^k \gamma_i v^i\right).
    \]
    Hence, $ c^\intercal (\sum_{i=1}^k \gamma_i v^i )\le 0 $.
    This yields
    \[
        c^\intercal z = c^\intercal \tilde{z} - c^\intercal \left(\sum_{i=1}^k \gamma_i v^i \right)\ge c^\intercal \tilde{z}.
    \]
    Since $\tilde{z}$ is optimal for ($J$-MIP), the latter inequality implies that $ z $ is also an optimal solution for ($ J $-MIP).
    The distance from $z$ to $ w $ can be bounded as follows:
    \[
        \|w-z\|_{\infty}
        = \| \sum_{i=1}^k (\lambda_i - \gamma_i) v^i \|_{\infty}
        \le \sum_{i=1}^k (\lambda_i - \gamma_i) \|v^i\|_{\infty}
        \le \sum_{i=1}^k (\lambda_i - \gamma_i) \Delta .
    \]

    It remains to argue that $ \sum_{i=1}^k (\lambda_i - \gamma_i) < d $.
    To this end, let us assume the contrary.
    Defining $ \alpha_i := \lambda_i - \gamma_i \ge 0 $ for each $ i \in [k] $, this means that $ \sum_{i=1}^k \alpha_i \ge d $.
    Thus, defining $ u^i \in \Z^d $ to be the projection of $ v^i $ onto the first $ d $ coordinates, we can invoke Lemma~\ref{lemIntFeasbility} to obtain $ \beta_1,\dotsc,\beta_k $ with $ \beta_i \in [0,\alpha_i] $ for each $ i \in [k] $ such that not all $ \beta_1,\dotsc,\beta_k $ are zero and $ \sum_{i=1}^k \beta_i u^i \in \Z^d $.
    Observe that $ \sum_{i=1}^k \beta_i v^i \in \Z^d \times \R^{n - d} $.
    For each $ i \in [n] $, define $ \gamma_i' := \gamma_i + \beta_i \ge 0 $ and note that $ \gamma_i' \le \gamma_i + \alpha_i = \lambda_i $.
    Furthermore, we have
    \[
        \sum_{i=1}^k \gamma_i' v^i = \underbrace{\sum_{i=1}^k \gamma_i v^i}_{\in\ \Z^d \times \R^{n - d}} + \underbrace{\sum_{i=1}^k \beta_i v^i}_{\in\ \Z^d \times \R^{n - d}} \in \Z^d \times \R^{n - d},
    \]
    and so $ (\gamma_1',\dotsc,\gamma_k') \in G $.
    However, since not all $ \beta_1,\dotsc,\beta_k $ are zero, $ \gamma'_1 + \dotsb + \gamma'_k > \gamma_1 + \dotsb + \gamma_k $, which contradicts the maximality of $ (\gamma_1,\dotsc,\gamma_k) $.
    \qed
\end{proof}


\section{The Davenport constant and the proof of Lemma~\ref{lemIntFeasbility}} \label{secDavenport}
We reduce the proof of Lemma~\ref{lemIntFeasbility} to a problem in additive combinatorics about the Davenport constant of certain abelian groups.
\begin{definition}[Davenport constant]
    Let $ G $ be a finite abelian (additive) group.
    The \emph{Davenport constant} of $ G $ is the smallest $ k \in \Z_{\ge 1} $ such that for every (not necessarily distinct) elements $ g^1, \dots, g^k\in G $, there exists a non-empty set $ I \subseteq [k] $ such that $ \sum_{i \in I} g^i = e $, where $ e $ is the identity element of $ G $.
\end{definition}
While determining the Davenport constant of a general abelian group is an open problem, Olson~\cite{Olson69} provided a tight answer for the case of so-called $ p $-groups.
A special case of his result reads as follows.
\begin{theorem}[Olson~\cite{Olson69}] \label{thmDavenport}
    Let $ d,p \in \Z_{\ge 1} $ with $ p $ prime.
    The Davenport constant of $ \Z^d / p\Z^d $ is $ pd - d + 1 $.
\end{theorem}
\begin{corollary} \label{corDavenport}
    Let $ d,p \in \Z_{\ge 1} $ with $ p $ prime.
    Let $ f^1, \dots, f^r \in \Z^d $ such that $ r \ge pd - d + 1 $.
    Then there exists a non-empty set $ I \subseteq [r] $ such that $ \sum_{i \in I} f^i \in p\Z^d $.
\end{corollary}
\begin{proof}[of Lemma~\ref{lemIntFeasbility}]
    By removing any vectors $u^i$ such that $\alpha_i =0$, we assume that $ \alpha_i > 0 $ for all $ i \in [k] $.
    We split the proof into two cases.

    First, assume that there exists a prime $ p $ such that for each $i\in [k]$, there is some $q_i \in \Z_{\ge 0}$ such that $ \alpha_i = q_i / p $.
    Consider the list
    \[
        \underbrace{u^1, \dots, u^1}_{q_1 \text{ copies}}, \underbrace{u^2, \dots, u^2}_{q_2 \text{ copies}}, \dotsc, \underbrace{u^k, \dots, u^k}_{q_k \text{ copies}}
    \]
    consisting of $ r := q_1 + \dotsb q_k = p (\alpha_1 + \dotsb + \alpha_k ) $ many vectors in $ \Z^d $.
    The inequality $\sum_{i=1}^k\alpha_i \ge d$ holds by assumption, so $ r \ge p d \ge pd - p + 1 $.
    Hence, by Corollary~\ref{corDavenport}, we obtain $ \ell_1,\dotsc,\ell_k $ with $ \ell_i \in \{0,\dotsc,q_i\} $ for $ i = 1,\dotsc,k $ such that not all $ \ell_1,\dotsc,\ell_k $ are zero and $ \sum_{i=1}^k \ell_i u^i \in p \Z^d $.
    Defining $ \beta_i := \ell_i / p $, we obtain $ \sum_{i=1}^k \beta_i u^i \in \Z^d $, where $ \beta_i \in [0,\alpha_i] $ for each $ i = 1,\dotsc,k $, and not all $ \beta_1, \dotsc, \beta_k $ are zero.
    The values $\beta_1, \dotsc, \beta_k$ prove the desired result in this first case.

    The case of general $ \alpha_1,\dotsc,\alpha_k $ is handled by a limit argument.
    The vector $ (\alpha_1,\dotsc,\alpha_k) $ can be approximated using fractions with prime denominators in the following way.
    For each $ j \in \Z_{\ge 1} $ there exists a prime $ p_j $ and integers $ q_{1,j},\dotsc,q_{k,j} \in \Z_{\ge 0} $ with $ q_{i,j}/p_j \in [\alpha_i, \alpha_i + 1/j] $ for all $ i \in [k] $.
    By construction,
    \[
        \lim_{j \to \infty} (q_{1,j} / p_j,\dotsc,q_{k,j} / p_j) = (\alpha_1,\dotsc,\alpha_k).
    \]
    By the previous case, for each $ j \in \Z_{\ge 1} $ there exist $ \beta_{1,j},\dotsc,\beta_{k,j} $ with $ \beta_{i,j} \in [0,q_{i,j}/p_j] $ such that not all $ \beta_{1,j},\dotsc,\beta_{k,j} $ are zero and $ \sum_{i=1}^k \beta_{i,j} u^i \in \Z^d $.
    Since the sequence $ (\beta_{1,j},\dotsc,\beta_{k,j}) $ ($ j \in \Z_{\ge 1} $) is contained in the compact set $ [0,\alpha_1 + 1] \times \dotsb \times [0,\alpha_k + 1] $, it contains a convergent subsequence.
    Thus, we assume that the limit
    \[
        \lim_{j \to \infty} (\beta_{1,j},\dotsc,\beta_{k,j}) =: (\beta_1,\dotsc,\beta_k)
    \]
    exists.
    For each $ i \in [k] $, the fact that $ \lim_{j \to \infty} q_{i,j}/p_j = \alpha_i $ together with $ \beta_{i,j} \in [0,q_{i,j}/p_j]$ for all $ j \in \Z_{\ge 1} $ implies that $ \beta_i \in [0,\alpha_i] $.
    Also, as there are only finitely many points in $ \Z^d $ of the form $ \sum_{i=1}^k \gamma_i u^i $ with $ \gamma_i \in [0, \alpha_i + 1]$ for each $ i \in [k]$, there exists some point $ z \in \Z^d $ such that
    \begin{equation}
        \label{eqFriday}
        \beta_{1,j} u^1 + \dotsb + \beta_{k,j} u^k = z
    \end{equation}
    holds for infinitely many $ j \in \Z_{\ge 1} $.
    This implies that $ \beta_1 u^1 + \dotsb + \beta_k u^k = z \in \Z^d $.
    If $\beta_1, \dots, \beta_k$ are not all zero, then they prove the desired result. 

    Otherwise, $\beta_1 = \dotsb = \beta_k = 0$, so $ z = 0 $.
    Choose any $ j \in \Z_{\ge 1} $ that satisfies~\eqref{eqFriday} and consider the vector $ (\varepsilon \beta_{1,j},\dotsc,\varepsilon\beta_{k,j}) $, where $ \varepsilon > 0 $ is chosen such that $ \varepsilon \beta_{i,j} \in [0,\alpha_i] $ holds for all $ i \in [k] $.
    Note that $\varepsilon$ exists since all $ \alpha_i $ are assumed to be positive.
    Not all components of $ (\varepsilon \beta_{1,j},\dotsc,\varepsilon\beta_{k,j}) $ are zero and 
    \[
        \varepsilon \beta_{1,j} u^1 + \dotsb + \varepsilon \beta_{k,j} u^k
        = \varepsilon (\beta_{1,j} u^1 + \dotsb + \beta_{k,j} u^k)
        = \varepsilon z
        = 0 \in \Z^d.
    \]
    Thus, the values $\varepsilon \beta_{1,j},\dotsc,\varepsilon\beta_{k,j}$ prove the desired result.
    \qed
\end{proof}


\section{Bounding distance in terms of $ \Delta $} \label{SecDelta}
We remark that all bounds discussed in this paper actually hold for arbitrary (not necessarily integer) right-hand sides $ b $.
A simple example given in~\cite[\S 17.2]{AS1986} shows that the bound of $ n \Delta $ is best-possible when comparing ($ \emptyset $-MIP) and ($ [n] $-MIP) for arbitrary $b$.
However, that example relies on the purely fractional components of $ b $, which may disappear after standard preprocessing of a linear integer program.
Assuming that $ b $ is integral, the following example shows that the distance depends at least linearly on $ \Delta $.
\begin{example}\label{exLowerBound}
For $ \delta \in \Z_{\ge 1} $, define
\[
A = \begin{bmatrix}
        -\delta &  0 \\
         \delta & -1
    \end{bmatrix},
    \quad
    b = \begin{bmatrix} -1 \\ 0 \end{bmatrix},
    \quad
    c = \begin{bmatrix} 0 \\ -1 \end{bmatrix}.
\]
Here, $ \Delta = \delta $.
The point $w = (1/\delta, 1)^\intercal$ is the unique optimal solution to solution of ($\emptyset$-MIP), and the point $z = (1, \delta)$ is the unique optimal solution of both ($\{1\}$-MIP) and ($\{1,2\}$-MIP).
For $J\in \{\{1\}, \{1,2\}\}$ and any optimal solution $w$ of ($\emptyset$-MIP), the closest optimal solution $z$ of ($J$-MIP) satisfies $\|z-w\|_{\infty} =  \delta -1 = \Delta -1$. \hspace*{\stretch{1}}$\diamond$
\end{example}
We are not aware of any pairs of MIPs for which distances between optimal solutions cannot be bounded just in terms of $ \Delta $.
For this reason, we believe that the distance bounds provided in this paper can be improved to a function that only depends on $ \Delta $, see Conjecture~\ref{conjMain}.
A case in which this conjecture holds is given by the following statement.
\begin{proposition} \label{propMain2}
    Assume that $ \Delta \le 2 $.
    Let $ I,J \in \{\emptyset, [n]\} $ such that~($ J $-MIP) has an optimal solution.
    For every optimal solution $ w $ of~($I$-MIP), there exists an optimal solution $ z $ of ($J$-MIP) such that $\|w-z\|_{\infty} \le \Delta$.
\end{proposition}
For the proof of Proposition~\ref{propMain2}, we use the following properties of so-called bimodular systems.
Unfortunately, similar results are unknown for matrices with $ \Delta \ge 3 $, and in future research, any similar results may be useful in extending Proposition~\ref{propMain2} to general $ \Delta $.
\begin{lemma}[{Veselov \and Chirkov~\cite[Thm. 2 and its proof]{VC2009}}] \label{lemVC}
    Let $ A \in \Z^{m \times n} $, $ b \in \Z^m $, $ c \in \R^n $ with $ \rank(A) = n $ such that the absolute value of any determinant of an $ n \times n $-submatrix of $ A $ is at most $ 2 $.
    Let $ x^* $ be a vertex of $ P := \{ x \in \R^n : Ax \le b \} $ and let $ Q $ be the convex hull of integer points satisfying all inequalities of $ Ax \le b $ that are tight at $ x^* $.
    Then
    \begin{enumerate}[label=(\alph*)]
        \item \label{VCedge} every vertex of $ Q $ lies on an edge of $ P $ that contains $ x^* $, and 
        \item \label{VCclose} every edge of $ P $ that contains $ x^* $ and some integer point, contains an integer point $ y^* $ with $ \|x^* - y^*\|_\infty \le 1 $.
    \end{enumerate}
\end{lemma}

\begin{proof}[of Proposition~\ref{propMain2}]
Let $ w \in \R^n $ be an optimal solution of ($I$-MIP) and let 
\[
P := \{ x \in \R^n : Ax \le b \} .
\]

We may assume that $ P $ is bounded.
Indeed, there exists some $U\in \Z_{\ge 1}$ such that the polytope
\(
P \cap \{x\in \R^n : -U \le x_i \le U~\forall i\in [n]\}
\)
contains $w$ and an optimal solution of ($J$-MIP).
It is sufficient to find an optimal solution $z$ of ($J$-MIP) contained in this polytope such that $\|w-z\|_{\infty}\le \Delta$. 
Also, the value of $\Delta$ for this polytope equals the value of $\Delta$ for $ P $.
Therefore,
by replacing $P$ with this polytope, we assume that $ P $ is bounded.
Since $P$ is non-empty and bounded, it follows that $\rank(A) = n$.
%
We split the remainder of the proof into four cases.
%

    \medskip
    
\noindent \emph{Case 1:} Assume that $ I = \emptyset $, $ J = [n] $, and $ w $ is a vertex of $ P $.

    \noindent
    Assume that $ w \in \Z^n$. It follows that $w$ is an optimal solution of ($[n]$-MIP). 
    Thus, setting $ z = w $ gives the desired bound $ \|w - z\|_\infty = 0 \le \Delta-1 \le \Delta$.

    \smallskip
    Assume that $w\not\in \Z^n$. 
    Since $w$ is a vertex of $P$, it must be the case that $ \Delta = 2 $ (otherwise, $A$ is totally unimodular, so $w\in \Z^n$).
    Let $ Q $ be defined as in Lemma~\ref{lemVC} and let $ z' \in \Z^n$ be a vertex of $ Q $ maximizing $ x \mapsto c^\intercal x $.
    By Lemma~\ref{lemVC}~\ref{VCedge}, $ z' $ lies on an edge $ E $ of $ P $ that contains $ w $.
    There is some $ z \in \Z^n \cap E$ such that $\|z-w\|_{\infty}\le \|\bar{z} - w\|_{\infty}$ for all $\bar{z}\in \Z^n\cap E$. 
    %
    The point $ z $ is in $ P $ and, by the optimality of $ w $ and $ z' $, it follows that $z$ is optimal for ($[n]$-MIP).
    By Lemma~\ref{lemVC}~\ref{VCclose}, we obtain the desired result $ \|w - z\|_\infty \le 1 \le \Delta -1 \le \Delta$.

    \smallskip
    Note that the optimal ($[n]$-MIP) solution $ z $ satisfies $ \|w - z\|_\infty \le \Delta - 1 $ in \emph{Case 1}.

    \medskip

\noindent \emph{Case 2:} Assume that $ I = \emptyset $ and $ J = [n] $.
    
    \noindent
    Let $ F \subseteq P $ be the face of all optimal solutions of ($\emptyset$-MIP) and let $ z' $ be an optimal solution of ($[n]$-MIP).
    Set $B:= \{x\in \R^n : \floor{w_i} \le x_i \le \ceil{w_i} ~ \forall i\in [n]\}$ and let $w'$ be a vertex of $B\cap F$.
    By construction of $B$, it follows that $\|w-w'\|_{\infty} \le 1$.

    Define the index sets
    \[
    \begin{array}{rcl}
    K_1 & := & \{i \in [n] : z_i \le \floor{w_i} \text{ and } w'_i = \floor{w_i} \},\\
    K_2 & := & \{i \in [n] : z_i \ge \floor{w_i} \text{ and } w'_i = \floor{w_i} \},\\
    K_3 & := & \{i \in [n] : z_i \le \ceil{w_i} \text{ and } w'_i = \ceil{w_i} \},\\
    K_4 & := & \{i \in [n] : z_i \ge \ceil{w_i} \text{ and } w'_i = \ceil{w_i} \},
    \end{array}    
    \]
    and the polytopes
    \[
    \begin{array}{rcl}
    P_1 & := & \{ x\in \R^n : x_i \le \floor{w_i} \text{ for all } i \in K_1\},\\
    P_2 & := & \{ x\in \R^n : x_i \ge \floor{w_i} \text{ for all } i \in K_2 \},\\
    P_3 & := & \{ x\in \R^n : x_i \le \ceil{w_i} \text{ for all } i \in K_3 \}, \text{ and }\\
    P_4 & := & \{ x\in \R^n : x_i \ge \ceil{w_i} \text{ for all } i \in K_4 \}.
    \end{array}    
    \]
    The polyhedron $\tilde{P}:= P \cap P_1 \cap P_2 \cap P_3 \cap P_4$ is non-empty, as it contains $w'$ and $z'$, and is bounded since it is contained in $P$, which itself is bounded. 
    Also, every inequality of $B\cap F$ that is tight at $w'$ is present (up to multiplication by $-1$) as an inequality defining $ \tilde{P} $. 
    Hence, $ w' $ is a vertex of $ \tilde{P} $.

    Note that $ \tilde{P} $ can be described by an integral inequality system whose coefficient matrix has rank equal to $ n $ and whose largest absolute value of a subdeterminant is equal to $\Delta$.
    Thus, by \emph{Case 1}, there is an integer point $ z \in \tilde{P} $ that maximizes $ x \mapsto c^\intercal x $ over $ \tilde{P} \cap \Z^n $ such that $ \|w' - z\|_\infty \le \Delta - 1 $.
    Since $ z' $ and $z$ are both in $ \tilde{P} $ and $z'$ is optimal for ($[n]$-MIP), the vector $ z $ is also optimal for ($[n]$-MIP).
    Furthermore, $ \|w - z\|_\infty \le \|w - w'\|_\infty + \|w' - z\|_\infty \le 1 + (\Delta - 1) = \Delta $.

    \medskip

\noindent \emph{Case 3:} Assume that $ I = [n] $, $ J = \emptyset $, and $ w $ is a vertex of $ \conv \{ x \in \Z^n : Ax \le b \} $.
   
    \noindent
    Assume that $ \Delta = 1 $. Hence, $A$ is a totally unimodular matrix, so $w$ is also an optimal solution of ($\emptyset$-MIP). 
    Setting $ z = w $, we obtain the desired result $ \|w - z\|_\infty = 0 \le \Delta$.

    Assume that $ \Delta = 2 $.
    The vector $w$ is a vertex of the polytope $R := \conv \{ x \in \Z^n : Ax \le b \} $, so there exists a vector $d\in \R^n$ such that 
    \(
    \{x\in R : d^\intercal x \ge d^\intercal \tilde{x}~\forall~ \tilde{x} \in R\} = \{w\}.
    \)   
    Let $F \subseteq P$ be the face of all optimal solutions of ($\emptyset$-MIP).
    Choose $ \lambda \ge 0 $ large enough so that there exists a vertex $ x^* \in F $ that maximizes $ x \mapsto (\lambda c + d)^\intercal x $ over $ P $.
    Setting $ \tilde{c} := \lambda c + d $, for every point $ x \in R \setminus \{w\} $ we have
    \[
        \tilde{c}^\intercal x
        = \lambda c^\intercal x + d^\intercal x
        < \lambda c^\intercal x + d^\intercal w
        \le \lambda c^\intercal w + d^\intercal w
        = \tilde{c}^\intercal w,
    \]
    where the first inequality follows from the definition of $ d $ and the second inequality from the optimality of $ w $.
    In other words, the point $ w $ is the unique maximizer of $ x \mapsto \tilde{c} x $ over $ R $.

    Given $x^*$, define $ Q $ as in Lemma~\ref{lemVC}.
    There exists a vertex $v$ of $Q$ that maximizes $x \mapsto \tilde{c} x$ over $Q$.
    Note that $v\in \Z^n$. 
    By Lemma~\ref{lemVC}~\ref{VCedge}, $v$ lies on an edge $E$ of $P$ that contains $x^*$. 
    Thus, $v$ maximizes $x \mapsto \tilde{c} x$ over $R$. 
    Since $w$ is the unique maximizer of this function over $R$, it follows that $w = v$. 
    In particular, $w$ lies on the edge $E$ of $P$ that contains $x^*$. 

    Now, consider again the objective function $x\mapsto c^\intercal x$. 
    If $c^\intercal x^* > c^\intercal w $, the open line segment from $ x^* $ to $ w $ does not contain integer points.
    Hence, by Lemma~\ref{lemVC}~\ref{VCclose}, $ \|x^* - w\|_{\infty} \le 1 $.
    Set $z = x^*$ to obtain the desired result $ \|z - w\|_{\infty} \le 1 \le \Delta$.
    If $c^\intercal x^* = c^\intercal w $, then $w$ is optimal for ($\emptyset$-MIP).
    Setting $ z = w $, we arrive at the desired result $ \|z - w\|_{\infty} = 0 \le \Delta $.

    \medskip

\noindent \emph{Case 4:} Assume that $ I = [n] $ and $ J = \emptyset $.
    
    \noindent 
    Since $ P $ is bounded, there exist vertices $ v^1, \dotsc,v^t $ of $ \conv \{ x \in \Z^n : Ax \le b \} $ and coefficients $ \lambda_1,\dotsc,\lambda_t > 0 $ with $ \lambda_1 + \dotsb + \lambda_t = 1 $ such that $ w = \sum_{j=1}^t \lambda_j v^j $.
    Note that $ v^1,\dotsc,v^t $ are all optimal solutions for ($[n]$-MIP).
    Thus, by \emph{Case 3}, for each $ j \in [t] $ there exists a point $ z^j $ that is optimal for ($\emptyset$-MIP) with $ \|v^j - z^j\|_\infty \le \Delta $.
    Define $ z := \sum_{j=1}^t \lambda_j z^j $.
    The point $z$ is also an optimal solution for the ($\emptyset$-MIP) and satisfies
    \begin{equation*}
        \|w - z\|_{\infty}
        = \big \| \sum_{j=1}^t \lambda_j (v^j - z^j) \big \| _\infty
         \le \sum_{j=1}^t \lambda_j \| v^j - z^j \|_\infty \\
         \le \sum_{j=1}^t \lambda_j \Delta
        = \Delta. 
    \end{equation*}
    \qed
\end{proof}


\section{Acknowledgements}

The authors would like to thank Ahmad Abdi for referring us to~\cite{AND2015}, where Olson's result was previously used. The second author was supported by the Alexander von Humboldt Foundation.


\bibliographystyle{spmpsci}
\bibliography{references}

\begin{thebibliography}{10}
\providecommand{\url}[1]{{#1}}
\providecommand{\urlprefix}{URL }
\expandafter\ifx\csname urlstyle\endcsname\relax
  \providecommand{\doi}[1]{DOI~\discretionary{}{}{}#1}\else
  \providecommand{\doi}{DOI~\discretionary{}{}{}\begingroup
  \urlstyle{rm}\Url}\fi

\bibitem{AND2015}
Andr\'as, M.: {M}atroid {M}etszetek {P}akol\'asa.
\newblock Master's thesis, E\"otv\"os Lor\'and Tudom\'anyegyetem
  Term\'eszettudom\'anyi Kar (2015)

\bibitem{BAL1995}
Baldick, R.: Refined proximity and sensitivity results in linearly constrained
  convex separable integer programming.
\newblock Linear Algebra and its Applications \textbf{226-228}, 389--407 (1995)

\bibitem{BJ1977}
Blair, C., Jeroslow, R.: The value function of a mixed integer program: I.
\newblock Discrete Mathematics \textbf{19}, 121--138 (1977)

\bibitem{BJ1979}
Blair, C., Jeroslow, R.: The value function of a mixed integer program: {II}.
\newblock Discrete Mathematics \textbf{25}, 7--19 (1979)

\bibitem{CGST1986}
Cook, W., Gerards, A., Schrijver, A., Tardos, E.: Sensitivity theorems in
  integer linear programming.
\newblock Mathematical Programming \textbf{34}, 251--264 (1986)

\bibitem{EW2018}
Eisenbrand, F., Weismantel, R.: Proximity results and faster algorithms for
  integer programming using the {S}teinitz lemma.
\newblock In: Proceedings of the Twenty-Ninth Annual ACM-SIAM Symposium on
  Discrete Algorithms, pp. 808--816 (2018)

\bibitem{GSK1988}
Granot, F., Skorin-Kapov, J.: Some proximity and sensitivity results in
  quadratic integer programming.
\newblock Mathematical Programming \textbf{47}(259-268) (1990)

\bibitem{HS1990}
Hochbaum, D.S., Shanthikumar, J.G.: Convex separable optimization is not much
  harder than linear optimization.
\newblock J. ACM \textbf{37}(4), 843--862 (1990).
\newblock \doi{10.1145/96559.96597}.
\newblock \urlprefix\url{http://doi.acm.org/10.1145/96559.96597}

\bibitem{Olson69}
Olson, J.E.: A combinatorial problem on finite abelian groups, {I}.
\newblock Journal of Number Theory \textbf{1}(8-10) (1969)

\bibitem{AS1986}
Schrijver, A.: Theory of linear and integer programming.
\newblock John Wiley \& Sons, Inc. New York, NY (1986)

\bibitem{VC2009}
Veselov, S., Chirkov, A.: Integer programming with bimodular matrix.
\newblock Discrete Optimization \textbf{6}, 220--222 (2009)

\bibitem{WM1991}
Werman, M., Magagnosc, D.: The relationship between integer and real solutions
  of constrained convex programming.
\newblock Mathematical Programming \textbf{51}, 133--135 (1991)

\end{thebibliography}
\end{document}